\documentclass[12pt,reqno]{amsart}
\usepackage{latexsym,amsmath,amsfonts,amscd,amssymb}
\usepackage{graphics}

\usepackage{tikz-cd}

\newcommand{\la}{\langle}
\newcommand{\ra}{\rangle}
\newcommand{\x}{\times}
\newcommand{\ox}{\otimes}

\newcommand{\ZZ}{\mathbb{Z}}
\newcommand{\CC}{\mathbb{C}}
\newcommand{\bk}{\mathbb{K}}
\newcommand{\CP}{\mathbb{C}P}
\newcommand{\RR}{\mathbb{R}}
\newcommand{\TT}{\mathbb{T}}
\newcommand{\QQ}{\mathbb{Q}}

\newcommand{\s}{\sigma}
\DeclareMathOperator{\Hom}{Hom}

\DeclareMathOperator{\Spec}{Spec}

\newcommand{\cO}{\mathcal O}
\newcommand{\PP}{\mathbb P}

\numberwithin{equation}{section}

\newtheorem{proposition}{Proposition}[section]
\newtheorem{theorem}[proposition]{Theorem}

\theoremstyle{definition}

\newtheorem{example}[proposition]{Example}
\newtheorem{remark}[proposition]{Remark}

\title[Rationally elliptic toric varieties]{Rationally elliptic toric varieties}

\author[I. Biswas]{Indranil Biswas}
\address{School of Mathematics, Tata Institute of Fundamental
Research, Homi Bhabha Road, Mumbai 400005, India}
\email{indranil@math.tifr.res.in}

\author[V. Mu\~{n}oz]{Vicente Mu\~{n}oz}
\address{Departamento de Algebra, Geometr\'{\i}a y Topolog\'{\i}a, Universidad de M\'alaga,
Campus de Teatinos, s/n, 29071 M\'alaga, Spain}
\email{vicente.munoz@uma.es}

\author[A. Murillo]{Aniceto Murillo}
\address{Departamento de Algebra, Geometr\'{\i}a y Topolog\'{\i}a, Universidad de M\'alaga,
Campus de Teatinos, s/n, 29071 M\'laga, Spain}
\email{aniceto@uma.es}

\subjclass[2010]{14M25, 55P62, 55Q52, 52B20}

\keywords{Complex toric variety, rational homotopy, elliptic space}

\begin{document}

\begin{abstract}
We give a characterization of all complete smooth toric varieties whose rational homotopy is of
elliptic type. All such toric varieties of complex dimension not greater than three are
explicitly described.
\end{abstract}

\maketitle

\section{Introduction}\label{sec:1}

Toric varieties have been widely studied from diverse points of view. Since they
have a combinatorial description in terms of polytopes
and many of their
topological properties (like the cohomology) are encoded combinatorially. Moreover,
they furnish a large source of examples of algebraic varieties.

In this paper, we are interested in the behaviour of the rank of the homotopy groups of
compact smooth toric varieties. Such a variety $X$ is always a formal algebraic manifold,
which means that its rational homotopy type depends only on its rational cohomology
$H^*(X,\,\QQ)$. In particular $\pi_*(X)\ox \QQ$, which is a rational vector space whose dimension is
precisely the total rank of $\pi_*(X)$, is explicitly determined by $H^*(X,\,\QQ)$.

It is known that a simply connected finite CW-complex $X$ is either {\em elliptic}, that is,
$\dim\pi_*(X)\otimes\QQ\,<\,\infty$, or it is {\em hyperbolic}, in which case
$\dim\pi_{\le k}(X)\otimes \QQ$ grows exponentially as $k$ increases.

Here we first notice that elliptic toric varieties are those whose rational cohomology
algebra is a complete intersection, that is, a polynomial algebra truncated by an ideal
generated by a regular sequence. Moreover, we prove that the Poincar\'e polynomial of these
toric varieties coincides with that of a product of complex projective spaces (see Theorem
\ref{main}). However, it may happen that the Poincar\'e polynomial of
a smooth toric variety $X$ is that of a product of complex projective spaces, but
$X$ is not elliptic (see Example \ref{example}(6)).

We illustrate the above result by describing all (algebraic isomorphism classes of) elliptic smooth toric varieties of dimension
less than or equal to $3$; see Theorems \ref{dos} and \ref{tres}. In dimension $2$, only
$\CP^2$ and the Hirzebruch surfaces $\PP\bigl(\cO_{\CP^1}\oplus \cO_{\CP^1}(b)\bigr)$ are
elliptic smooth toric varieties. Their homotopy types are $\CP^2$, $\CP^1\x \CP^1$ and $\CP^2
\# \overline{\CP^2}$.

The $3$-dimensional elliptic smooth toric varieties are, up to isomorphism, $\CP^3$,
$\PP(\cO_{\CP^2}\oplus \cO_{\CP^2}(c))$, $\PP(\cO_{\CP^1}\oplus \cO_{\CP^1}(a)
\oplus \cO_{\CP^1}(b))$, and
$\CP^1$-bundles over Hirzebruch surfaces.
These varieties have the rational homotopy type of $\CP^3\#\CP^3$, $\CP^1\times \CP^2$ and
a quotient $(S^3\times S^3\times S^3)/ T^3$ respectively.

Elliptic complex manifolds of complex dimension less than or equal to $2$ have been classified in
\cite{AB}. Also, the rational homotopy type of moduli spaces of certain vector bundles over complex curves
has been analysed in \cite{Bi-Mu}. Our results here complement these references and show, in particular,
the existence of a very large number of (non isomorphic) hyperbolic varieties.

We finally stress that our results and classification on toric varieties are always up to algebraic isomorphisms. If one
loosens up this rigidity there are interesting results in the literature. A toric variety is in particular a torus
manifold, that is, a smooth $2n$-dimensional manifold admitting an action of a real $n$-torus $\TT^n=(S^1)^n$. In
\cite[Theorem 1.1]{wie}, it is shown that an elliptic simply connected torus manifold whose integral cohomology is evenly graded is
always homeomorphic to a quotient of a free linear torus action on a product of spheres. Moreover \cite[Theorem 1.2]{wie},
if the toric manifold is non-negatively curved then homeomorphism can be replaced by equivariant diffeomorphism. More
generally, see \cite[Theorem A]{wie2}, any elliptic simply connected torus orbifold has the rational homotopy type of a
quotient of a product of spheres by a linear, almost free, torus action. Also, T. Bahri has drawn our attention to
\cite{Gu} for related results.

In low dimensions one can be more precise: any $4$-dimensional torus manifold is
equivariantly diffeomorphic to either the $4$-sphere, an equivariant connected sum of
copies of complex projective planes (possibly with reversed orientation) and Hirzebruch surfaces \cite{or}.
As mentioned above, the only ones that are rationally elliptic are the $4$-sphere, the complex projective plane and the Hirzebruch surfaces.
On the other hand, a simply connected $6$-dimensional torus manifold whose cohomology is 
evenly graded is equivariantly diffeomorphic to either the $6$-dimensional sphere, an equivariant connected sum
of copies of $6$-dimensional quasitoric manifolds, or $S^4$-bundles over $S^2$, see \cite[Theorem 1.3]{ku}.
Our results agree with this classification up to diffeomorphism.

After the completion of this work, Wiemeler has provided a classification of rationally elliptic toric orbifolds
up to algebraic isomorphism in any dimension \cite{Wieme}.

\subsection*{Acknowledgements}
We thank Matthias Franz for warning us about the formality issue of compact toric varieties,
Michael Wiemeler for pointing us to references on torus manifolds, and Tony Bahri for referring us to \cite{Gu}.
We are grateful to the two referees for their useful comments.
The first author is supported by a J. C. Bose Fellowship, and school of mathematics,
TIFR, is supported by 12-R$\&$D-TFR-5.01-0500.
The second author was partially supported by the MICINN grant MTM2015-63612-P (Spain).
The third author was partially supported by the MICINN grant MTM2016-78647-P (Spain).

\section{Toric varieties and rational homotopy}

We recall some results from rational homotopy theory that will be used, and also summarize
a brief introduction to toric varieties. More details can be found in \cite{FHT} and
\cite{Fulton} respectively.

A {\em toric variety} is a complex algebraic variety $X$
of complex dimension $N$ equipped with an algebraic action of the complex
torus $T\,=\,(\CC^*)^N$ such that $X$ contains a dense $T$-orbit on which the
action of $T$ is free.

Any toric variety can be described by a fan as follows. Let $\Gamma$ be a lattice, meaning
a group isomorphic to $\ZZ^N$, and let $\Delta$ be a fan of $\Gamma$, that is, a collection of
strongly convex rational polyhedral cones in the vector space $V=\Gamma \ox_\ZZ \RR$,
satisfying the conditions of a simplicial complex: every face of a cone in $\Delta$ is also a
cone in $\Delta$, and the intersection of two cones in $\Delta$ is a face of each of the cones.
Recall that a strongly convex rational polyhedral cone $\sigma$ in $F$ is a cone with apex at
the origin, generated by a finite number of vectors in the lattice, and which intersects the
opposite cone $-\sigma$ only at the apex. A fan is complete if the union of its conforming
cones is $V$.

Let $M\,=\,\Hom(\Gamma,\,\ZZ)$ be the dual lattice. Every cone $\sigma$ determines a
finitely generated commutative semigroup
 $$
 S_\s\,=\,\{ u\,\in\, M \, \mid \, \la u,\,v\ra \,\geq\, 0, ~\text{ for all }~ v\,\in\,\s\},
 $$
and an associated affine variety
 $$
 U_\s \,=\,\Spec (\CC[S_\s])\, .
 $$
Any face $\tau$ of a cone $\s$ in a given fan $\Delta$ induces an inclusion $S_\s\,\subset\,
S_\tau$ which, in turn, produces an open embedding $U_\tau \,
\longrightarrow\, U_\s$. All these affine varieties fit together to form an
algebraic variety which denoted by $X(\Delta)$.
As the apex $(0)$ is a face of every cone and $S_{(0)}\,=\,M$, it follows that $$U_{(0)}\,=\,
\Spec(\CC[x_1,x_1^{-1}, \ldots , x_N,x_N^{-1}])\,=\, (\CC^*)^N\,=\,T$$ is contained as
an open subset of
$X(\Delta)$. Moreover, for every cone $\sigma$ of $\Delta$, the
group $T$ acts on $U_\sigma$ via the map induced by the diagonal $\CC[S_\s] \,\longrightarrow\,
\CC[S_\s]\otimes \CC[S_{(0)}]$. These also fit together to produce an action of $T$ on $X(\Delta)$.

A toric variety is compact if and only if its generating fan is complete \cite[\S3.5]{da}.
The criterion for smoothness is that for any collection of $N$ spanning vectors
$v_1,\ldots, v_N$ of the fan, the following holds:
$$|\det(v_1,\ldots, v_N)|\,=\, 1\, ,$$ that is, they span $\ZZ^N$ (cf.\ \cite[page 29]{Fulton}).
Note that the sign of the determinant
is positive if the basis $\{v_1, \ldots, v_N\}$ is oriented.
From now on we shall only consider compact and smooth toric varieties.
Recall that these toric varieties are all simply connected \cite[\S3.2]{Fulton}.

Concerning the rational homotopy theory, any topological space considered here shall be (of the
homotopy type of) a simply connected CW-complex. Two such spaces $X$ and $Y$ have the same
rational homotopy type if there is a continuous map $f\,\colon\, X\,\longrightarrow\, Y$ such that
$\pi_*(f)\otimes
\QQ\colon\pi_*(X)\otimes\QQ\stackrel{\cong}{\longrightarrow}\pi_*(Y)\otimes\QQ$ is an
isomorphism. Such a map $f$ is called a rational homotopy equivalence. A space $X$ is {\em formal} if
its rational homotopy type depends only on its rational singular cohomology algebra
$H^*(X,\,\QQ)$. To give a more precise definition of this property the following notion is indispensable, see \cite[\S12]{FHT}:

A {\em Sullivan model} of a given a commutative differential graded algebra (cdga henceforth) $A$ is a cdga quasi-isomorphism of the form
$$
(\Lambda V,d)\stackrel{\simeq}{\longrightarrow} A.
$$
where $\Lambda V$ denotes the free commutative algebra generated by the graded vector space $V$ and whose differential $d$
satisfies a special recurrence property: there is a well ordered basis $\{v_\alpha\}$ on $V$ such that, for each $\alpha$, $dv_\alpha$ is a ``polynomial'' in $\Lambda V$ which only involves the generators $\{v_\beta\}_{\beta<\alpha}$. If $A$ is the cdga $\Omega_{PL}(X)$ of ``polynomial forms'' on a given space $X$, this is a {\em Sullivan model of $X$}. Whenever $X$ is a manifold, $\Omega_{PL}(X)$ can be replaced by the classical de Rham forms $\Omega(X)$.

As a Sullivan model of a given cdga characterizes its homotopy type,
a space $X$ is {\em formal} if a Sullivan model of $H^*(X;\,\QQ)$ (with trivial differential) is also a Sullivan model of $X$.
Classical examples of formal spaces are compact K\"ahler
manifolds \cite{DGMS}. On the other hand, a formal space $X$ is said to be {\em intrinsically formal} if it is the only
(up to rational homotopy equivalence) space with $H^*(X,\,\QQ)$ as rational cohomology algebra.

The {\em elliptic-hyperbolic dichotomy} \cite[\S33]{FHT}
asserts that given a simply finite CW-complex $X$ then, either its homotopy groups have finite total rank,
$$
\dim\pi_*(X)\otimes \QQ\,<\, \infty\, ,
$$
or else, $\dim\pi_i(X)\otimes \QQ$ grows exponentially as $i$ increases: there exists $\lambda>1$ and an integer $n$ such that
$$
\sum_{i\le k}\pi_i(X)\otimes \QQ\ge \lambda^k,\quad \text{for $k\ge n$}.
$$
In the first case the space $X$ is said to be {\em elliptic}; otherwise, it is {\em hyperbolic}.

\section{Ellipticity of toric varieties}

We first check formality of compact toric varieties. This is not automatic as they are not K\"ahler in general: indeed, smooth 
compact toric varieties are not necessarily projective (see \cite{Ew} for examples). On the other hand, since $H^2(X)$ is 
generated by divisors, it follows that $H^2(X)=H^{1,1}(X)$. Thus, if $X$ is K\"ahler then it is projective, and the assertion 
follows. For completeness, we remark that quasitoric manifolds, the topological analogue of projective toric varieties, are 
also known to be formal \cite[Corollary 7.2]{paray}.

\begin{proposition}\label{prop::3.1}
A smooth compact toric variety $X$ is formal.
\end{proposition}

\begin{proof} We first recall the explicit description of the integer cohomology ring of a smooth toric variety $X$.
Let $D_1,\dots,D_d$ be the {\em irreducible $T$-divisors} of $X$, that
is, the submanifolds of complex codimension $1$ which are invariant under the
action of the complex torus $T$ \cite[\S3.3]{Fulton}. Each $T$-divisor $D_i$ correspond to
an edge ($1$-cone) of the fan $\Delta$, and each of them is determined by the first point $v_i$ in the lattice which is touched by the edge. More concretely, $D_i=U_{v_i}-U_{(0)}$.

The cohomology of a compact smooth toric variety is:
\begin{equation}\label{ecuacion}
H^*(X)\,=\,\ZZ[D_1,\ldots, D_d]/I\, ,
\end{equation}
where each $D_i$ is of degree $2$ and $I$ is the ideal generated by the following:
\begin{enumerate}
\item[(i)] The products $D_{i_1}\cdots D_{i_k}$, for $v_{i_1},\ldots, v_{i_k}$ distinct
vertices which do not lie in a cone of $\Delta$, and

\item[(ii)] $\sum_{i=1}^d \la u_j,\,v_i\ra D_i$, with $\{u_j\}$ a basis of $M$.
\end{enumerate}
This is proved in \cite[page 106]{Fulton} for projective toric varieties. 
The statement for any compact smooth toric variety follows from \cite{da}.
%

Now, we inductively build a Sullivan model of $X$
 by means of this cohomology description.
First, for each generator $D_i$ and each relation in (ii) fix a Thom form $\eta_i\,\in \,\Omega^2(X)$ representing $D_i$,
 and for each $u_j$ we fix a $1$-form $\xi_j\in\Omega^1(X)$ such that $d\xi_j=\sum_{i=1}^d \la u_j,\,v_i\ra \eta_i$.
Consider the graded vector space $V_1=\la w_i, w_j\ra$,
where $|w_i|=2$, $|w_j|=1$, set $dw_i=0$, $dw_j=\sum_{i=1}^d \la u,\,v_i\ra w_i$ and define
$$
f\colon (\Lambda V_1,d)\longrightarrow \Omega (X)
$$
by $f(w_i)\,=\,\eta_i$, $f(w_j)\,=\,\xi_j$. Obviously $H(f)$ is surjective and $H^2(f)$ is an isomorphism.

Next, observe that if
$v_{i_1},\ldots, v_{i_k}$ are distinct vertices which do not lie in a cone of the fan,
then the intersection of the divisors $D_{i_1},\ldots, D_{i_k}$ is empty and thus, $\eta_{i_1}\wedge\ldots
\wedge \eta_{i_k}=0$. Hence, the kernel of $f$ is generated by the corresponding $w_{i_1}\cdots w_{i_k}$.
Then, define the graded vector space $V_2=\la w_I\ra$, where $I$ runs over the tuples $\{i_1,\ldots, i_k\}$
such that $v_{i_1},\ldots, v_{i_k}$ do not lie in a cone, declare
$dw_I=w_{i_1}\cdots w_{i_k}$ and extend $f$ to
$$
f\colon ( \Lambda (V_1\oplus V_2),d)\longrightarrow \Omega (X),
$$
by setting $f(V_2)=0$. Again, $H(f)$ may fail to be injective as non-trivial kernel may appear with the new generators. Hence,
we continue this process to find a quasi-isomorphism
$$
f\colon(\Lambda V,d)\stackrel{\simeq}{\longrightarrow}\Omega(X)
$$
in which $V=\oplus_{n\ge 1} V_n$, $dV_n\subset \Lambda (V_1\oplus\dots\oplus V_{n-1})$ and $f(V_k)=0$, for $k\ge 2$. By construction, this is a Sullivan model of $X$.

Finally, observe that $f$ readily produces also a quasi-isomorphism
$$
\widetilde f\colon (\Lambda V,d)\stackrel{\simeq}{\longrightarrow} (H^*(X),0)
$$
which is trivial on any generator except $\widetilde f(w_i)=D_i$. This proves the formality.
\end{proof}

\begin{remark}
The proof of Proposition \ref{prop::3.1} generalizes to locally standard torus manifolds with integral cohomology
generated in degree $2$. This is because by \cite{MP06} there is a similar presentation of the cohomology
ring as in the case of toric manifolds.
\end{remark}

In what follows, and as usual, the Betti numbers of a given space are $b_k(X)\,=\,\dim H^k(X,\,\QQ)$, the Euler
characteristic of $X$ is $\chi(X)\,=\,\sum_k (-1)^k b_k$, and its Poincar\'e polynomial is given by
 $$
 P_X(t)\,=\,\sum_k b_k t^k\, .
 $$

As for any local Noetherian ring, recall that a polynomial algebra $\bk[x_1,\dots,x_n]/I$ over
a field $\bk$ is a {\em complete intersection} if $I$ is generated by a {\em regular sequence}
$p_1,\dots,p_m$. That is, for each $i=2,\dots,n$, the class of each $p_i$ is not a zero
divisor in $\bk[x_1,\dots,x_n]/(p_1,\ldots,p_{i-1})$. If this polynomial algebra is finite
dimensional, it is a complete intersection if and only if $m=n$.

\begin{theorem} \label{main}
A smooth compact toric variety $X$ is elliptic if and only if its cohomology algebra
$H^*(X,\,\QQ)$ is a complete intersection concentrated in even degrees.

When the above condition is satisfied, $X$ is intrinsically formal, and its Poincar\'e polynomial
coincides with that of a product of complex projective spaces.
\end{theorem}

\begin{proof} By the general description in (\ref{ecuacion}), $H^*(X)$ is evenly graded.
Assume that $X$ is elliptic. By \cite[Proposition 32.16]{FHT}, the cohomology $H^*(X,\,\QQ)$
is evenly graded if and only if it is of the form $\QQ[x_1,\dots,x_n]/(p_1,\dots,p_n)$ in
which every $x_i$ is of even degree while $p_1,\,\ldots,\,p_n$ is a regular sequence.

Conversely, it is well-known (see for instance \cite[\S3]{feha}) that any space whose
cohomology algebra is a complete intersection concentrated in even degrees is intrinsically
formal and elliptic.

Let $X$ be an elliptic toric variety whose rational cohomology is
the complete intersection $\QQ[x_1,\dots,x_n]/(p_1,\dots,p_n)$ in which each $x_i$ is of
even degree. Again by \cite[Proposition 32.16]{FHT}, this has two other equivalent
re-formulations:
\begin{enumerate}
\item $\chi(X)>0$, and

\item $\dim\pi_{\text{even}}\otimes\QQ=\dim\pi_{\text{odd}}\otimes\QQ\,=\,n$.
\end{enumerate}
Moreover, if $2\alpha_1,\,\cdots,\, 2\alpha_n$ and $2\beta_1-1,\,\cdots,\,2\beta_n-1$ are the
degrees of a basis of $\pi_*(X)\otimes \QQ$,
then $2\alpha_i$ is precisely the degree of $x_i$ for all $i$, and the Poincar\'e polynomial of $X$ is given by
$$
 P_X(t)=\frac{\prod_{i=1}^n(1-t^{2\beta_i})}{\prod_{i=1}^n(1-t^{2\alpha_i})}\,.
 $$

Now, since $H^*(X,\,\QQ)$ is generated by elements of degree $2$, every $\alpha_i=1$ and
 \begin{equation}\label{eqn:Pt}
 P_X(t)=\frac{\prod_{i=1}^n(1-t^{2\beta_i})}{(1-t^{2})^n}\,.
 \end{equation}
But the Poincar\'e polynomial of the projective space $\CP^k$ is
 $$
 P_{\CP^k}(t)= \frac{1-t^{2k+2}}{1-t^2}= 1+t^2+t^4+\ldots +t^{2k}\,,
$$
and thus, the Poincar\'e polynomial of $X$ is the same as that of
 $$
 \CP^{\beta_1-1} \times \ldots\times \CP^{\beta_n-1}\, .
 $$
In particular, we have $\dim_\CC X\,=\,\beta_1+\ldots + \beta_n -n$.
\end{proof}

\begin{remark}
The proof of Theorem \ref{main} generalizes to manifolds with cohomology generated in degree $2$.
\end{remark}

Note that Theorem \ref{main} implies in particular that, for an elliptic toric variety $X$ of (complex) dimension $N$, we have
\begin{equation}\label{eqn:b2}
b_2 \,=\, n \,\leq\, \sum (\beta_i-1) \,=\,N\, .
\end{equation}
This is a classical fact, which is valid in general for elliptic $1$-connected finite CW-complexes.

\begin{example}\label{example}\mbox{}
\begin{enumerate}
\item The projective space $\CP^N$ is a toric variety. The torus $T=(\CC^*)^N$ acts by $(t_1,\ldots,t_N)\cdot
[z_0,z_1,\ldots,z_N]=[z_0,t_1z_1,\ldots, t_Nz_N]$. Clearly $\CP^N$ is elliptic.

\item Since the product of toric varieties is a toric variety, the polynomial (\ref{eqn:Pt}) is the Poincar\'e polynomial of (at least one) elliptic toric variety, namely $ \CP^{\beta_1-1} \times \ldots\times \CP^{\beta_n-1}$.

\item If $X$ is a toric variety, and $Y\,\subset\, X$ is a $T$-invariant subvariety, then the
blow-up $X'\,=\,\mathrm{Bl}_Y X$
of $X$ along $Y$ is again a toric variety. One particular example of this is the blow-up
of a $T$-fixed point.

For instance, take $X\,=\,\CP^2$ and blow-up at a $T$-fixed point $p$. Then $X'\,=\,
\mathrm{Bl}_pX$ is a toric variety, and
as a $C^\infty$ manifold it is $X\cong \CP^2\#\overline{\CP^2}$, where
$\overline{\CP^2}$ means $\CP^2$ with the opposite orientation. The cohomology of $X'$ is
 \begin{equation}\label{eqn:CP2}
 H^*(X',\,\QQ)\,=\,\QQ[x,y]/(xy, \,x^2+y^2)\, ,
 \end{equation}
where $x$ is the cohomology class of a line in $\CP^2$, $y$ the cohomology class of the exceptional divisor, and
$x^2$ is the volume form. The cohomology algebra (\ref{eqn:CP2}) is a complete intersection, so by Theorem \ref{main},
the manifold $X'$ is elliptic.

This also follows from \cite[Theorem 1.1]{AB} as $\CP^2\#\overline{\CP^2}$ is diffeomorphic
to a Hirzebruch
surface; it also follows from the more general result, \cite[Lemma 3.2]{pape}, that
classifies the homeomorphic type
of simply connected elliptic closed real manifolds of dimension $4$.

\item Consider a toric variety $Y$ with equivariant line bundles $L_0,\ldots, L_k$. Then the
projectivization of the total space of the bundle $L_0\oplus \ldots\oplus L_k \longrightarrow Y$ is
again a toric variety \cite[page 42]{Fulton}; denote $X\,:=\,\PP(L_0\oplus \ldots \oplus L_k)$. Moreover, if $Y$
is elliptic then $X$ is also elliptic. This follows from the fact that topologically $X$ is a
fibration
with fiber $\CP^k$ and base $Y$, and hence it has finite-dimensional total rational homotopy.
Furthermore, the Poincar\'e polynomial of $X$ is $P_X(t)\,=\,P_Y(t) P_{\CP^k}(t)$; this follows from
\cite{Delig}.

In particular, taking $Y\,=\,\CP^1$, we have the Hirzebruch surfaces $X_b=\PP(\cO_{\CP^1}\oplus \cO_{\CP^1}(b))$,
and these are (again) elliptic and toric. Topologically, for $b$ even we have a diffeomorphism
$X_b\cong \CP^1\x \CP^1$, and for $b$
odd we have $X_b\cong \CP^2\#\overline{\CP^2}$. These share the same Poincar\'e polynomial, but have different
cohomology algebras.
So they have different homotopy type although they have isomorphic homotopy groups.

\item Let $X$ be the blow-up of $\CP^N$ at a point. As a smooth manifold $X\cong \CP^N\# \overline{\CP^N}$.
Recall also that $\overline{\CP^N}\cong \CP^N$ for $N$ odd, since such $\CP^N$ admits an orientation reversing
diffeomorphism. The cohomology is
$$
 H^*(X,\,\QQ)\,=\,\QQ[x,y]/(xy, x^N+y^N),
$$
where $x$ is the class of the hyperplane of $\CP^N$ and $y$ is the class of the exceptional divisor.
This algebra is a complete intersection, and thus $X$ is elliptic.

The variety $X$ can also be described as a $\CP^1$-bundle over $\CP^{N-1}$. More concretely,
$X\,=\,\PP(\cO_{\CP^{N-1}} \oplus\cO_{\CP^{N-1}}(1))$. It follows again that it is toric and elliptic.

\item Consider the toric variety $X$ given as the blow-up of $\CP^3$ at two fixed points. As a smooth
manifold $X\cong \CP^3\#\overline{\CP^3}\#\overline{\CP^3} \cong \CP^3\# \CP^3\# \CP^3$. Then,
 $$
 H^*(X,\,\QQ)\,=\,\QQ[x,y,z]/(xy,xz,yz,x^3-y^3,x^3-z^3)\,.
 $$
Here, $x,y,z$ are the generators of $H^2(\CP^3,\, \QQ)$ for each of the three connected summands.
This algebra is not a complete intersection (by Poincar\'e duality it needs at least three relations in degree $4$ and
another two relations in degree $6$). Thus, by Theorem \ref{main}, the manifold $X$ is hyperbolic.
However, its Poincar\'e polynomial is
$P_X(t)\,=\,(1+t^2)^3$ and it coincides with that of $\CP^1\times\CP^1\times \CP^1$. This example shows that the Poincar\'e polynomial does not characterizes the rational homotopy type of a compact smooth toric variety.
Also $b_2=3$, so it satisfies the inequality in (\ref{eqn:b2}) even though $X$ is not elliptic.
\end{enumerate}
\end{example}

\begin{remark}
It is straightforward to check that any Poincar\'e duality algebra $H$ generated in degree $2$ and concentrated in degrees 
$\le\, 4$ is necessarily a complete intersection. In particular, as observed in the proof of Theorem \ref{main}, such an 
algebra is intrinsically formal and thus, it is the rational cohomology of a unique, up to rational homotopy, space $X$ whose 
Euler homotopy characteristic is zero. That is, $\dim\pi_{\text{even}}X\otimes\QQ\,=\,\dim\pi_{\text{odd}}X\otimes\QQ$. We then may 
apply \cite[Theorem 1.2]{mishi} to conclude that there are infinitely many different isomorphism classes of these algebras, 
namely:
$$
\QQ[x]/(x^3),\qquad \QQ[x,y]/(x^2,y^2),\qquad \QQ[x,y]/(x^2+\lambda y^2,xy),\,\,\lambda\in\QQ^*/(\QQ^*)^2.
$$
As stated, any of these algebras determines a unique and distinct rational
homotopy type. However only the following four of them
$$
\QQ[x]/(x^3),\qquad \QQ[x,y]/(x^2,y^2),\qquad \QQ[x,y]/(x^2\pm y^2,xy),
$$
are realized by manifolds (see \cite[Lemma 3.2]{pape}), namely by
 $$
\CP^2,\quad \CP^1\times \CP^1,\quad \CP^2\#\CP^2,\quad \CP^2\#\overline{\CP^2}.
$$
The third one is not a complex manifold (from the Enriques-Kodaira classification \cite{BPV}). All
the other three are toric by Example \ref{example}.
\end{remark}

On the other hand, Example \ref{example}(6) exhibits a hyperbolic toric variety $X$ of complex
dimension $3$ with $P_X(t)=(1+t^2)^3$. Such examples do not exist for the other possible choice
of the Poincar\'e polynomial in the same dimension given by Theorem \ref{main}, namely
the polynomial $(1+t^2)(1+t^2+t^4)$.

\begin{proposition}\label{propo} Let $H$ be a Poincar\'e duality algebra generated in
degree $2$ with $P_H(t)=(1+t^2)(1+t^2+t^4)$. Then:
 \begin{enumerate}
 \item[(i)] $H$ is a complete intersection.
 \item[(ii)] There are countable different rational homotopy types of elliptic simply connected CW-complexes whose cohomology is of this kind.
\item[(iii)] The only manifolds with cohomology of this kind are $\CP^1\times \CP^2$ and\break $\CP^3\#\CP^3$.
 \end{enumerate}
\end{proposition}

\begin{proof} \noindent (i) Choose a free presentation of $H$ which is necessarily of the form
$H=\QQ[x,y]/I$, for some ideal $I$. In view of the given Poincar\'e polynomial, we conclude
that $I$ must contain only one quadratic polynomial $p$ and only one cubic polynomial $q$ not
generated by $p$. We will show that $I=(p,q)$, from which it follows that $H$ is a compete intersection.

We complexify $H_\CC=H\otimes_\QQ \CC =\CC[x,y]/I$. It is enough to show that $I=(p,q)$ in
$\CC[x,y]$. Factor $p=p_1p_2$, $q=q_1q_2q_3$, and arrange variables (by a linear change of
variables) so that $p_1=x$. If $p_1=p_2$ then, unless some $q_i=x$, all the monomials
$x^4,x^3y,x^2y^2, xy^3,y^4$ are contained in $I$ and thus $I=(p,q)$. If $p_1\not=p_2$ then we may
assume $p_2=y$ and again, unless one of the $q_i$ equals either $x$ or $y$, all the monomials
of degree $4$ are in $I$, which is necessarily generated by $p,q$.

We finish the proof of (i) by showing that none of the $q_i$ can match $x$ or $y$. By
contradiction, we assume without losing generality that $q_1=p_1=x$. As $H_\CC$ is Poincar\'e
duality, let $r$ be a quadratic polynomial so that $q_1r\notin I$. If either $q_2$ or $q_3$
coincides with $p_2$, then $q$ is a multiple of $p$ which contradicts our hypothesis. If both
$q_2$ and $q_3$ are different from $p_2$, then any quadratic polynomial, in particular $r$, is
generated by $p_2,q_2q_3$. Therefore, $q_1r$ is in the ideal generated by $p_1p_2,q_1q_2q_3$
which is $I$ and we again reach a contradiction.

\medskip \noindent (ii) Complete intersection algebras $H$ with $\dim H<8$ are classified in
\cite[Theorem 1.2]{mishi}. An inspection shows that any such algebra as in the statement, i.e.,
generated in degree $2$ and with the prescribed Poincar\'e polynomial, is necessarily of the
form $$ \QQ[x,y]/(x^2+\lambda y^2,\mu x^3+\gamma x^2y), $$ where the rational numbers
$\lambda,\mu,\gamma$ run through a precise countable set. As any of this algebras is
intrinsically formal, it is the cohomology algebra of exactly one elliptic space, up to
rational homotopy.

\medskip \noindent (iii) By \cite[Theorem 1.3]{herr}, cf.\ \cite[Theorem 1.3]{AB},
the rational homotopy types in (ii) which can be realized by a manifold are $\CP^1\times \CP^2$ and $\CP^3\#\CP^3$.
\end{proof}

\section{Elliptic toric varieties of dimension $2$ and $3$}

 There are various
classifications of toric varieties in terms of their describing fans. In the general case,
a $d$-dimensional toric variety with second Betti number $b_2$ is given by a polytope in $\RR^d$ with $d+b_2$ spanning
vertices. As $b_2$ can be arbitrarily large for toric varieties, there are infinitely many isomorphism classes
of toric varieties of given dimension $d$. Among those, algebraic geometers have paid special attention on classifying {\em Fano} toric varieties, those whose anti-canonical line bundle is ample. Apart
from other geometric considerations, the list of these varieties is finite as the Fano property produces a bound of $b_2$.
Precisely, there are
$5, 18, 124, 866$ isomorphism classes of Fano $d$-dimensional toric varieties for $d=2,3,4,5$ (see \cite{Adv}).

However, from the topological point of view, restricting to the Fano property is most unnatural. Nevertheless, the bound (\ref{eqn:b2}) on $b_2$ will let us, in particular,
attack the classification of elliptic toric varieties in low dimensions.

Notice first that $\CP^1$ is the only toric variety of dimension $1$. In dimensions $2$ and $3$ we have:

\begin{theorem}\label{dos} Any elliptic smooth toric variety of dimension $2$ is either $\CP^2$ or a Hirzebruch surface $\PP\bigl(\cO_{\CP^1}\oplus \cO_{\CP^1}(b)\bigr)$ of invariant $b$.
\end{theorem}

\begin{proof} Let $X$ be a $2$-dimensional, elliptic toric variety and keep in mind that $b_2(X)=b_2\le 2$ and that the
generating fan is necessary complete as $X$ is smooth.
For $b_2=1$, let $v_1,v_2,v_3$ be vectors in $\RR^2$ spanning the generating fan of $X$, which we label in
the cyclic order.
We can arrange coordinates so that $v_1\,=\,(1,\,0)$ and $v_2\,=\,(0,\,1)$, since any two vectors of these vectors constitute a basis of $\ZZ^2$
by smoothness of $X$. Again by smoothness,
$\det(v_3,\,v_1)\,=\,\det (v_2,\,v_3)\,=\,1$ and therefore $v_3\,=\,(-1,\,-1)$. Hence $X\,=\,\CP^2$ which is elliptic.

For $b_2\,=\,2$, let the fan of $X$ be spanned by four vectors in $\RR^2$, $v_1,v_2,v_3,v_4$ that we write in
cyclic order. Therefore we may consider $v_1\,=\,(1,\,0),\, v_2\,=\,(0,\,1)$ and
$\det(v_2,\,v_3)\,=\,\det (v_3,\,v_4)\,=\,\det(v_4,\,v_1)\,=\,1$,
again by the smoothness of $X$. The defining matrix is given by
 $$
\left(\begin{array}{cccc} 1 & 0 & -1 & a \\
 0 & 1 & b & -1 \end{array}\right),
 $$
with $ab\,=\,0$, $a,b\,\in\, \ZZ$. For $a\,=\,0$ the matrix is
 $$
\left(\begin{array}{cccc} 1 & 0 & -1 & 0 \\
 0 & 1 & b & -1 \end{array}\right),
 $$
which corresponds to the Hirzebruch
surface $X_b=\PP(\cO_{\CP^1}\oplus \cO_{\CP^1}(b))$ of invariant $b$. Topologically, it is the fiber bundle over $\CP^1$ whose fiber is $\CP^1$ and
the Chern class is $c_1=b$. This is always rationally elliptic as, in general, the total space of a fibration in which the base and the fiber are elliptic, is also elliptic. For $b=0$, $a\neq 0$, we can swap the
coordinates (and the order of $v_1,v_2$ and of $v_3,v_4$) to go back to the previous case.

Observe that, with the notation in (\ref{ecuacion}),
$$
H^*(X_b,\,\QQ)\,=\, \QQ[D_1,D_2,D_3,D_4]/I\, ,
 $$
 where
 $$
 I\,=\,(D_1D_3, D_2D_4, D_1-D_3,D_2+bD_3-D_4)\, .
 $$
Setting $D_1\,=\,x,\, D_2\,=\,y, \,D_3\,=\,x,\, D_4\,=\,y+bx$ we obtain that the cohomology is:
 $$
H^*(X_b,\,\QQ)\,=\, \QQ[x,y]/(x^2, y(y+bx))\, .
 $$
Here $x$ is the class of the fiber, and $y$ is the class $\cO_X(1)$.
\end{proof}

Recall that there is a diffeomorphism $X_b\cong \CP^1\x \CP^1$ for $b$ even, and $X_b \cong \CP^2\# \overline{\CP^2}$
for $b$ odd. Hence there are three smooth manifolds (and three rational homotopy types) corresponding to the
varieties in Theorem \ref{dos}, namely: $\CP^2$, $\CP^1\x \CP^1$ and $\CP^2\# \overline{\CP^2}$.

\begin{theorem}\label{tres} Let $X$ be an elliptic smooth toric variety of dimension $3$.
 \begin{enumerate}
\item If $b_2\,=\,1$, then $X\,=\,\CP^3$.
\item If $b_2\,=\,2$, then $X$ is either $\PP(\cO_{\CP^2}\oplus \cO_{\CP^2}(c))$ or $\PP(\cO_{\CP^1}\oplus \cO_{\CP^1}(a) \oplus \cO_{\CP^1}(b))$ and thus, it has the rational homotopy type of $\CP^3\#\CP^3$ or $\CP^1\times \CP^2$ respectively.
\item If $b_2\,=\,3$, then $X$ is a $\CP^1$-bundle over a Hirzebruch surface and has the rational homotopy
type of a
quotient $(S^3\times S^3\times S^3)/ T^3$.
\end{enumerate}
\end{theorem}

\begin{proof} We follow the notations and
descriptions of \cite{WW}. For computing the cohomology of the resulting varieties we use repeatedly its presentation in (\ref{ecuacion}). Again, in view of the inequality (\ref{eqn:b2}), $b_2\leq 3$.

\medskip \noindent (1) For $b_2=1$, the generating fan is spanned by four vectors which can always be chosen as
$v_1=(1,0,0), v_2=(0,1,0), v_3=(0,0,1),v_4=(-1,-1,-1)$. This produces $X=\CP^3$.

\medskip \noindent (2) For $b_2=2$, choose the triangulation of the fan in \cite[page 41]{WW} with (oriented) cones $(v_1,v_2,v_3)$,
 $(v_1,v_3,v_4)$, $(v_1,v_4,v_5)$, $(v_1,v_5,v_2)$, $(v_2,v_4,v_3)$, $(v_2,v_5,v_4)$. Recall that
$\det(v_i,v_j,v_k)=1$ for each cone, by smoothness of $X$. We can arrange that
$v_1=(1,0,0), v_2=(0,1,0), v_3=(0,0,1)$. This produces the matrix
 $$
\left(\begin{array}{ccccc} 1 & 0 & 0 & -1 & a \\
 0 & 1 & 0 & -1 & b \\
 0 & 0 & 1 & c & -1 \end{array}\right),
 $$
where $ac=bc=0$, $a,b,c\in \ZZ$. Therefore there are two possibilities:
\begin{equation}\label{eqn:myfan}
\left(\begin{array}{ccccc} 1 & 0 & 0 & -1 & 0 \\
 0 & 1 & 0 & -1 & 0 \\
 0 & 0 & 1 & c & -1 \end{array}\right), \qquad
\left(\begin{array}{ccccc} 1 & 0 & 0 & -1 & a \\
 0 & 1 & 0 & -1 & b \\
 0 & 0 & 1 & 0 & -1 \end{array}\right).
 \end{equation}

The first fan in (\ref{eqn:myfan})
produces $X\,=\,\PP(\cO_{\CP^2}\oplus \cO_{\CP^2}(c))$ which is a $\CP^1$-bundle over
$\CP^2$, always rationally elliptic with Poincar\'e polynomial $(1+t^2)(1+t^2+t^4)$.
A computation similar to the one in the proof of Theorem \ref{dos} shows that
 $$
H^*(X,\,\QQ)\,=\, \QQ[x,y]/(x^3,y(y+cx))\, .
$$
Thus, as a smooth manifold, using Proposition \ref{propo}, $X$ has the rational homotopy type of
(and hence it is diffeomorphic to) $\CP^3\#\CP^3$ if $c\,\not=\,0$, and that of $\CP^1\x\CP^2$ if $c=0$.

The second fan in (\ref{eqn:myfan}) corresponds to
$X\,=\,\PP(\cO_{\CP^1}\oplus \cO_{\CP^1}(a) \oplus \cO_{\CP^1}(b))$,
which is a $\CP^2$-bundle over $\CP^1$ and again elliptic. We also obtain that
$$
H^*(X,\,\QQ)\,=\,\QQ[x,y]/(x^2,y(y+ax)(y+bx))
 =\QQ[x,y]/(x^2,y^3-\lambda x y^2 ),
 $$
for suitable $\lambda$, and therefore $X$ has always the rational homotopy type of (and
hence it is diffeomorphic to) $\CP^2\x \CP^1$.

\medskip \noindent (3) If $b_2=3$, by \cite[page 42]{WW}, there are two possible triangulations of the generating fan of $X$:

\medskip \noindent Case (I). The (oriented) cones are $(v_1,v_2,v_3)$,
$(v_2,v_6,v_3)$, $(v_1,v_5,v_2)$, $(v_2,v_5,v_6)$, $(v_1,v_4,v_5)$, $(v_1,v_3,v_4)$,
$(v_3,v_6,v_4)$, $(v_4,v_6,v_5)$. This produces the matrix
 $$
\left(\begin{array}{cccccc} 1 & 0 & 0 & a & d & -1 \\
 0 & 1 & 0 & -1 & c & f \\
 0 & 0 & 1 & b & -1 & e \end{array}\right),
 $$
where $bc=de=af=0$ and $ace=-dfb$, $a,b,c,d,e,f\in \ZZ$. There are six cases:
$a=c=d=0$, $a=b=e=0$, $c=e=f=0$, $a=b=d=0$, $c=d=f=0$, $b=e=f=0$.
In all cases, there is at least one column with $0$ in two of the entries. We reorder variables so that this happens in the
last position, that is, $v_6=(-1,0,0)$. We also reorder the first and second coordinates
so that $c=0$. This gives the
 $$
\left(\begin{array}{cccccc} 1 & 0 & 0 & a & d & -1 \\
 0 & 1 & 0 & -1 & 0 & 0 \\
 0 & 0 & 1 & b & -1 & 0 \end{array}\right).
 $$

A computation shows that
$$
H^*(X,\,\QQ)\,=\, \QQ[D_1,D_2,D_3,D_4,D_5,D_6]/I\, .
 $$
where
$$
I=(D_2D_4, D_3D_5, D_1D_6,D_1-aD_4+dD_5-D_6,D_2-D_4,D_3+bD_4-D_5).
$$
 Choosing $D_4=D_2=x$, $D_5=y$, $D_3=y-bx$, $D_6=z$ and $D_1=z-ax-dy$, we obtain
$$
H^*(X,\,\QQ)\,= \,\QQ[x,y,z]/(x^2, y(y-bx), z(z-ax-dy))\, ,
$$
which is the cohomology algebra of a $\CP^1$-bundle over the Hirzebruch surface
$\PP(\cO_{\CP^1}\oplus \cO_{\CP^1}(b))$, associated to the Chern class $ax+dy$
(where $x$ is the class of the fiber and $y$ is the class of the section).
In particular, $X$ is rationally elliptic. Moreover, being intrinsically formal,
it has the rational homotopy type of the
quotient $(S^3\times S^3\times S^3)/ T^3$, via the action
 $$
 (u,v,w)\cdot \bigl((p_1,p_2),(q_1,q_2),(r_1,r_2)\bigr)\,=\,\bigl((up_1,up_2),(uq_1,u^bvq_2),(ur_1,u^av^dwr_2)\bigr),
 $$
in which $(u,v,w)\,\in\, T^3$ and $\bigl((p_1,p_2),(q_1,q_2),(r_1,r_2)\bigr)\in S^3\times S^3\times S^3\subset\CC^2\times\CC^2\times \CC^2$.
Indeed, see \cite[Proposition 4.26]{de}, such manifold has precisely the above rational cohomology algebra.

\medskip \noindent Case (II). The (oriented) cones are $(v_1,v_2,v_3)$,
$(v_1,v_3,v_4)$, $(v_3,v_5,v_4)$, $(v_1,v_4,v_6)$, $(v_1,v_6,v_2)$, $(v_4,v_5,v_6)$,
$(v_2,v_6,v_3)$, $(v_3,v_6,v_5)$. The solutions are given by:
{\scriptsize $$
{\scriptstyle \left(\begin{array}{cccccc} 1 & 0 & 0 & a & a-1 & -1 \\
 0 & 1 & 0 & -1 & -1 & 0 \\
 0 & 0 & 1 & b & b & -1 \end{array}\right),
\left(\begin{array}{cccccc} 1 & 0 & 0 & 1 & a & -1 \\
 0 & 1 & 0 & -1 & -a-1 & 1 \\
 0 & 0 & 1 & 0 & b &-1 \end{array}\right),
\left(\begin{array}{cccccc} 1 & 0 & 0 & a & -1 & -1 \\
 0 & 1 & 0 & -1 & 0 & 1 \\
 0 & 0 & 1 & 0 & 0 &-1 \end{array}\right)
,}
$$}
{\scriptsize $$
{\scriptstyle \left(\begin{array}{cccccc} 1 & 0 & 0 & 1 & 0 & -1 \\
 0 & 1 & 0 & -1 & -1 & a \\
 0 & 0 & 1 & 0 & 0 & -1 \end{array}\right),
\left(\begin{array}{cccccc} 1 & 0 & 0 & 0 & -1 & -1 \\
 0 & 1 & 0 & -1 & a-1 & a \\
 0 & 0 & 1 & 0 & 0 &-1 \end{array}\right),
\left(\begin{array}{cccccc} 1 & 0 & 0 & n & m & -1 \\
 0 & 1 & 0 & -1 & -1-mp & p \\
 0 & 0 & 1 & 0 & 0 &-1 \end{array}\right)
,}
$$}

\noindent Here, the first five families have parameters $a,b\in \ZZ$, and the sixth matrix
corresponds to $5$ isolated cases, with parameters $(n,m,p)=(-1,1,-3)$, $(-2,1,-2)$, $(-1,2,-2)$, $(-3,2,-1)$, $(-2,3,-1)$,
which are solutions to $n(1+mp)-m=1$ not appearing in the previous families.

The cohomology of is given this time by
 $$
 \QQ[D_1,D_2,D_3,D_4,D_5,D_6]/I.
 $$
where the common non linear generators of $I$ are,
$$
D_1D_5, D_2D_4, D_2D_5, D_3D_4D_6, D_1D_3D_6
$$
and the linear ones are immediately obtained in view of the corresponding matrix.
Now we take $x=D_4, y=D_5, z=D_6$, to get that the cohomology rings in every one of the six cases is:
\begin{align*}
&\QQ[x,y,z]/(x^2+xy,y^2+xy,y^2+yz,xz^2, (1-a-b)y^3+z^3),\\
&\QQ[x,y,z]/(xy-yz, (x+(a+1)y-z)x, y^2, (x-by)z^2,(z-(a+b)y)z^2), \\
&\QQ[x,y,z]/((1-a)xy+y^2, x^2-xz, xy-yz, x^3, yz^2+z^3), \\
&\QQ[x,y,z]/(xy-yz,x^2+xy -axz, (1-a)xy+y^2, xz^2, z^3),\\
&\QQ[x,y,z]/(y^2+yz, x^2-axz +(1-a) xy, xy+y^2, xz^2, yz^2+z^3),\\
&\QQ[x,y,z]/((n+mnp-m)xy-yz,\\
&\qquad\qquad\qquad (x+(1+mp)y-pz)x,(1-np)xy+y^2, xz^2, z^3-myz^2).
\end{align*}
We now check that none of
the quotient ideals in this list is generated by a regular sequence and thus,
the corresponding toric variety is always hyperbolic. As the argument we follow is similar for all of them we only do the first one. As a vector space, the algebra is generated by
 $$
 1, x, y,z , xy, xz, z^2, x^2y, xyz, y^3, x^2yz ,
 $$
which gives the Poincar\'e polynomial $P_X(t)=1+3t^2+3t^4+t^6=(1+t^2)^3$. Hence,
for the ideal of relations $I$, we need at least three relations of degree $4$ (quadratic on $x,y,z$).
But then $xz^2 \not\in (x^2+xy,y^2+xy,y^2+yz)$, so we need at least another relation of degree $6$
showing that the cohomology ring $\QQ[x,y,z]/I$ is not a complete intersection, and hence the toric variety
is hyperbolic.
\end{proof}

\begin{remark}\mbox{}
\begin{itemize}
\item[(i)] The above process can potentially be carried out for dimension $d=4$, although the number
of cases grows drastically. An alternative possibility is to restrict to finding which Fano toric varieties
of dimensions $d\,=\,4,\,5$ are elliptic, using the classifications in \cite{Adv}. Note that the Fano condition
implies that there are finitely many and a bound of $b_2$.

\item[(ii)] The toric varieties described in Case (II) of the proof of Theorem \ref{tres} constitute a large class of hyperbolic manifolds of very special nature.
\end{itemize}
\end{remark}

\end{document}